\documentclass[12pt]{article}

\usepackage{tikz}
\usetikzlibrary{arrows,shapes.geometric,positioning}
\usepackage{todonotes}
\usepackage{pgf}

\usepackage[all]{xy}
\usepackage[ruled,vlined]{algorithm2e}
\usepackage{wrapfig}
\usepackage{yfonts}
\usepackage[T1]{fontenc}
\usepackage{array}
\usepackage{arydshln}
\usepackage{multirow}

\usepackage{epic}
\usepackage{amsmath}[1996/11/01]
\usepackage{amssymb,amsthm,amsfonts,latexsym,epsfig,graphics}
 \def\beql#1#2\eeql{\begin{equation}\label{#1}#2\end{equation}}

\advance \textheight by -1 \topmargin
\advance \textheight by -1 \headheight
\advance \textheight by -1 \headsep
\advance \textheight by -2 \footskip
\vsize = \textheight
\hsize = \textwidth
\pgfdeclareimage[height=6cm]{rla1}{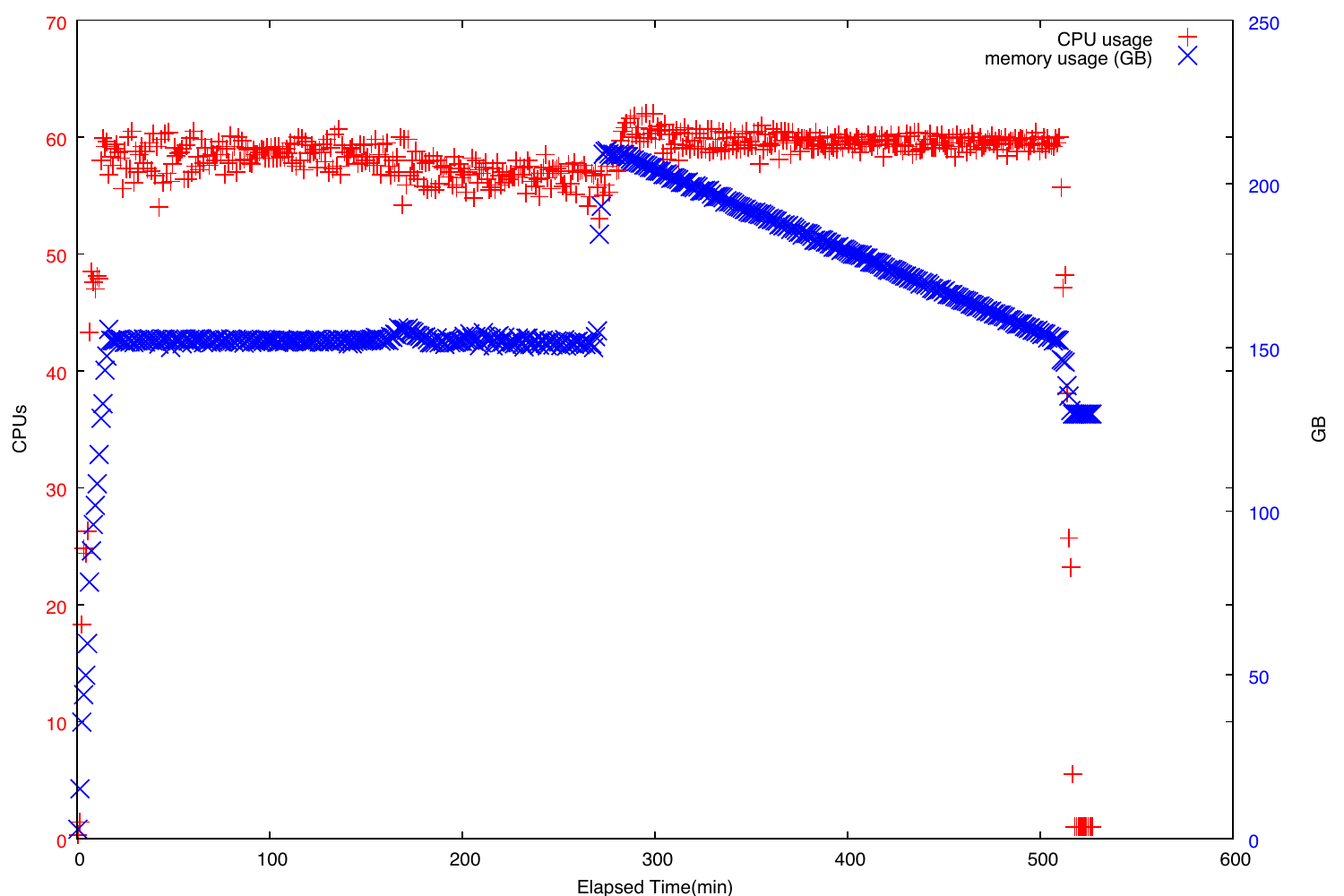}
\pgfdeclareimage[height=5cm]{ch10}{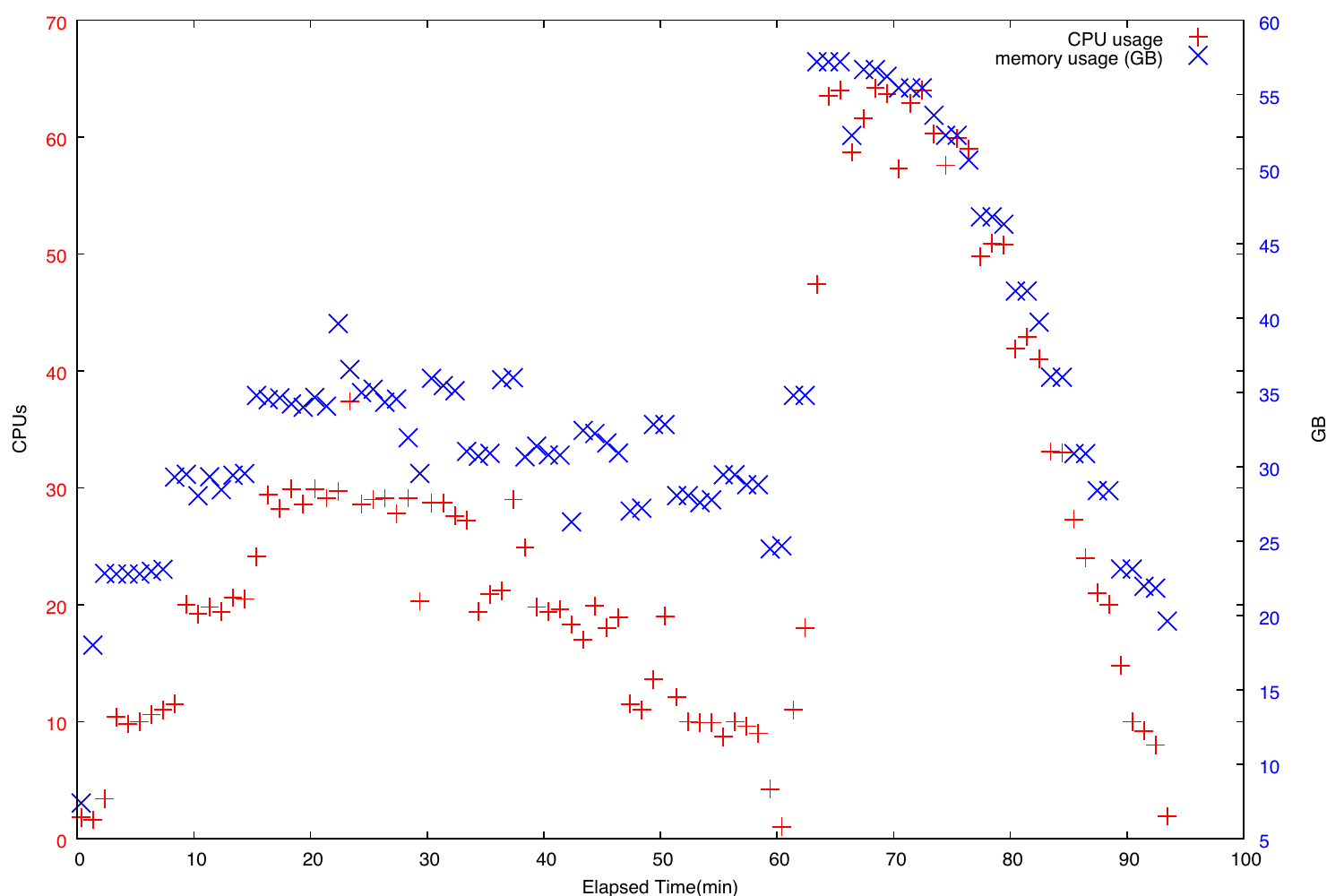}
\pgfdeclareimage[height=5cm]{ch20}{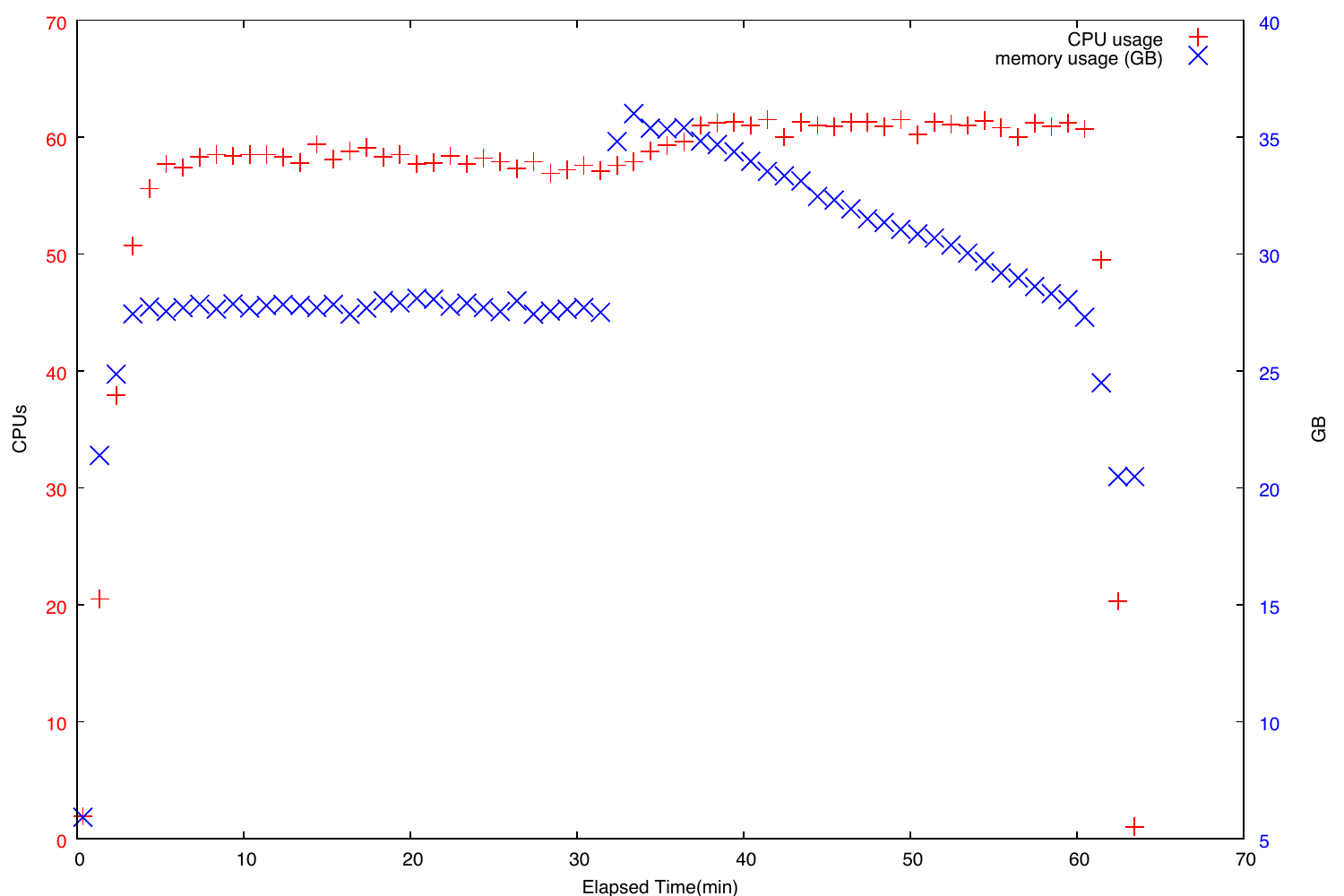}

\DeclareMathOperator{\ech}{ECH}
\DeclareMathOperator{\PC0}{UN0}
\DeclareMathOperator{\PVC}{UNH}
\DeclareMathOperator{\MKR}{MKR}

\DeclareMathOperator{\mad}{MAD}
\DeclareMathOperator{\mul}{MUL}
\DeclareMathOperator{\RRF}{RRF}

\DeclareMathOperator{\GL}{GL}

\DeclareMathOperator{\CC}{{\text{\sc ClearDown}}}
\DeclareMathOperator{\PCU}{{\text{\sc PreClearUp}}}
\DeclareMathOperator{\CU}{{\text{\sc ClearUp}}}
\DeclareMathOperator{\UR}{{\text{\sc UpdateRow}}}
\DeclareMathOperator{\URT}{{\text{\sc UpdateRowTrafo}}}

\DeclareMathOperator{\CRZ}{CRZ}
\DeclareMathOperator{\TT}{\mathcal{M}}
\DeclareMathOperator{\KK}{\mathcal{K}}

\DeclareMathOperator{\diag}{diag}

\DeclareMathOperator{\RA}{{\bf \rightarrow}}
\DeclareMathOperator{\DA}{{\bf \downarrow}}

\newtheorem{theorem}{Theorem}[section]

\newtheorem{remark}[theorem]{Remark}
\newtheorem{lemma}[theorem]{Lemma}
\newtheorem{proposition}[theorem]{Proposition}
\newtheorem{definition}[theorem]{Definition}

\newcommand{\F}{{\mathbb{F}}}

\newcommand{\R}{{\mathbb{R}}}
\newcommand{\A}{{\mbox{\bf{A}}}}
\newcommand{\B}{{\mbox{\bf{B}}}}
\newcommand{\C}{{\mbox{\bf{C}}}}
\newcommand{\D}{{\mbox{\bf{D}}}}
\newcommand{\E}{{\mbox{\bf{E}}}}
\newcommand{\K}{{\mbox{\bf{K}}}}
\newcommand{\M}{{\mbox{\bf{M}}}}
\newcommand{\X}{{\mbox{\bf{X}}}}

\renewcommand{\R}{{\mbox{\bf{R}}}}

\title{A parallel algorithm for Gaussian elimination over finite fields}
\author{Stephen Linton, Gabriele Nebe, Alice Niemeyer,  Richard
  Parker, Jon Thackray} 
\begin{document}
\maketitle

\textsc{
  School of Computer Science, University of St.\ Andrews, St.\ Andrews,
  Fife KY169SX, Scotland}
\emph{E-mail address}{:\;\;}\texttt{steve.linton@st-andrews.ac.uk}
\medskip 

\textsc{
Lehrstuhl D f\"ur Mathematik, RWTH Aachen University,
52056 Aachen, Germany}
\emph{E-mail address}{:\;\;}\texttt{nebe@math.rwth-aachen.de}

\medskip 
\textsc{
Lehrstuhl B f\"ur Mathematik, RWTH Aachen University,
52056 Aachen, Germany}
\emph{E-mail address}{:\;\;}\texttt{alice.niemeyer@mathb.rwth-aachen.de}
\medskip 

\textsc{70 York St. Cambridge CB1 2PY, UK}\\
\emph{E-mail address}{:\;\;}\texttt{richpark54@hotmail.co.uk}
\medskip 

\emph{E-mail address}{:\;\;}\texttt{jgt@pobox.com}

\begin{abstract}
In this  paper we describe  a parallel Gaussian  elimination algorithm
for  matrices  with  entries  in   a  finite  field.  Unlike  previous
approaches, our  algorithm subdivides a  very large input  matrix into
smaller submatrices by subdividing both  rows and columns into roughly
square  blocks  sized so  that  computing  with individual  blocks  on
individual  processors provides  adequate concurrency.   The algorithm
also  returns  the  transformation   matrix,  which  encodes  the  row
operations  used.   We  go  to  some  lengths  to  avoid  storing  any
unnecessary data as we keep track of the row operations, such as block
columns of the transformation matrix known to be zero.

The algorithm  is accompanied  by a  concurrency analysis  which shows
that the improvement in concurrency is  of the same order of magnitude
as the number  of blocks. An implementation of the  algorithm has been
tested on matrices as large as $1  000 000\times 1 000 000$ over small
finite fields.
\end{abstract}






\section{Introduction}
\label{sec:intro}

Already  employed for  solving equations  by hand  in China  over 2000
years ago~\cite{History},  the Gaussian elimination method  has become
an invaluable tool in many  areas of science.  When computing machines
became  available,  this  algorithm  was   one  of  the  first  to  be
implemented on a computer.  In general  the algorithm takes as input a
matrix with  entries in a  field (or  a division ring)  and transforms
this matrix to  a matrix in row  echelon form. It can  be employed for
several  different  purposes,  and  computer  implementations  can  be
tailored to suit the intended  application.  Different variants of the
Gaussian elimination algorithm can be envisaged, for example computing
the rank of  a matrix, computing a  row echelon form or  a reduced row
echelon form  of a  matrix, or  computing one  of these  echelon forms
together with the transformation matrix.   Often one of these versions
of  the Gaussian  elimination algorithm  lies  at the  heart of  other
algorithms for  solving problems in a  broad range of areas  and their
overall  performance  is often  dictated  by  the performance  of  the
underlying Gaussian elimination algorithm.   Thus an implementation of
a Gaussian  elimination algorithm  is required to  display exceptional
performance.

Since their invention, computers have  become faster and more powerful
every year.  Yet, for over a decade this increase in computing power
is no longer primarily due to faster CPUs but rather to the number of different
processors an individual computer has, paired with the increasingly
sophisticated memory hierarchy. It is therefore paramount that modern
algorithms are tailored to modern computers. In particular this means
that they need to be able to perform computations in parallel and store
the data for the current computations readily in cache.

With the  advance of  parallel computers  comes the  need to  design a
parallel algorithm to perform Gaussian  elimination on a matrix.  Such
a parallel algorithm  would immediately result in  immense speedups of
higher  level  algorithms  calling the  Gaussian  elimination  without
having to introduce parallelism  to these algorithms themselves.  When
designing a parallel Gaussian elimination algorithm it is important to
keep the applications of the  algorithms in mind.  Several versions of
a  parallel Gaussian  elimination algorithm  have been  described when
working over  the field of  real or  complex numbers, see  
\cite{realgausssurvey} for a survey and 
PLASMA  \cite{plasma} for  implementations.
In this paper we describe a parallel version of the
Gaussian elimination algorithm which, given a matrix with entries in a
finite field,  computes a reduced  row echelon form together  with the
transformation matrix. We  note that when working  with dense matrices
over finite fields  we are not concerned  with sparsity,
the selection of suitable pivot elements nor numerical accuracy.
In particular, we can 
always choose the first non-zero element of a given row as our
pivot element, ensuring that  all entries to the left of a
pivot are known to be zero. Moreover, we need not be concerned with
producing elements in a field which require more and more memory to
store them. Avoiding producing very large field elements
would  again complicate pivot selection.
Thus our main concern is
to design a parallel algorithm which makes optimal use of modern
parallel computers.

We assume an underlying shared  memory computational model in which we
have access to  $k$ different processors, each of which  can run a job
independently  from any  other. The  processors communicate  with each
other through the  shared memory system. Our design  must take account
of the limited total memory bandwidth  of the system.  The aim of our
parallel algorithm is to achieve  adequate concurrency by dividing the
necessary computational work into smaller jobs and scheduling these to
run simultaneously  on the $k$ processors.   This in turn calls  for a
very careful  organization of the jobs  so that different jobs  do not
interfere  with  each   other.   We  will  address   these  issues  in
Section~\ref{sec:model}.

It is well known, see for example \cite[Theorems~28.7,~28.8]{cormen},
that  the asymptotic complexities of matrix  inversion and 
matrix multiplication are equal. An  algorithm that shows
this reduces inversion of a  $2n\times 2n$ matrix to 6 multiplications
of $n\times n$ matrices together with two inversions, also of $n\times
n$ matrices.  Applying this  approach recursively,  almost all  of the
run-time of  inversion is spent  in multiplications of  various sizes.
It is  not difficult  to see  that this extends  to our  somewhat more
general computation of  a reduced row echelon  form, with transformation
matrix.

We envisage that we are given a very large matrix over a finite
field for which we need to compute a reduced row echelon form together
with a transformation matrix. 

Several approaches to achieving this  already exist for finite fields.
One  approach  takes  advantage  of the  reduction  to  multiplication
mentioned  above,  and  delegates  the problem  of  parallelising  the
computation primarily to the much  easier problem of parallelising the
larger  multiplications.   Another  approach represents  finite  field
elements as floating point real numbers of various sizes in such a way
that (with care)  the exact result over the finite  field can still be
recovered. The  problem can then  be delegated to  any of a  number of
highly  efficient  parallel  floating  point  linear  algebra  systems
\cite{fflas}. Another approach by Albrecht \emph{et al.} works 
over small finite fields of characteristic 2 (see
\cite{ffmul} and \cite{albreachtweb}).
A  fourth  approach  in
unpublished   work  by   L\"ubeck   repeatedly   divides  the   matrix
horizontally, echelonising  each block  of rows  in parallel  and then
sorting  the rows  of the  matrix, so  that rows  with similar  length
initial sequences of zeros come together.

Large modern computers typically have a  large number of cores but may
well have  an order of  magnitude less  real memory bandwidth  per core
than a  typical laptop  or desktop  computer. On  such a  large modern
computer, at  the lowest level, a  core can only be  fully occupied if
essentially all  its data is in  the smallest, fastest level  of cache
memory (L1). At the  next level out, this work can  only be started if
essentially all its data is in the next level of cache (L2). A similar
statement is true for L3 cache.   To use a modern computer effectively
for  matrix  operations,  it  is  therefore  necessary  to  repeatedly
subdivide matrices in {\bf both} directions producing matrices that
are roughly square 
at a scale commensurate  with the size of the cache
at each level.

It is  not too hard to  design an algorithm for  matrix multiplication
with  these properties.  An  approach along  these  lines to  Gaussian
elimination  for  these roughly  square  submatrices  is described  in
Section~\ref{sub:echmult}.

For the whole matrix, we also need to subdivide
to achieve concurrency.

Unlike previous approaches,  we subdivide our very  large input matrix
into smaller  submatrices, called  \emph{blocks}, by  subdividing both
rows and  columns into roughly  square blocks sized so  that computing
with  individual blocks  on  individual  processors provides  adequate
concurrency.  We will  show that we gain a  concurrency improvement in
the  same   order  of   magnitude  as  the   number  of   blocks,  see
Proposition~\ref{prop:concur} and Theorem~\ref{the:concur}.

As well as computing the reduced row echelon form of the input matrix,
we  compute the  \emph{transformation matrix},  which encodes  the row
operations  used.   We  go  to  some  lengths  to  avoid  storing  any
unnecessary data as we keep track of the row operations, such as block
columns  of  the   transformation  matrix  known  to   be  zero.   Our
experiments show  that the  memory usage during  the execution  of the
algorithm  remains  fairly stable,  and  is  similar  to  storing  the  input
matrix.  The  runtime  is  broadly  comparable  to  multiplication  of
matrices to  the same size as  the input matrix.  This  gives evidence
that  we  have  succeeded  in   keeping  the  cost  of  computing  the
transformation matrix as small as possible and is in
accordance with the theoretical analysis of B\"urgisser et al.\, in
particular \cite[Theorem~16.12]{Buerg}.

The parallel Gauss algorithm is designed with three distinct environments in
view, although we have only implemented the first so far.

The first (and original) target is to use a single machine with
multiple cores with the matrix in shared memory. 
Here the objective is to subdivide
the overall task into many subtasks that can run concurrently. 

The second target is to use a single machine where the matrix does not
fit in memory, but does fit on disk. Here the objective is to
subdivide the matrix so that each piece fits into memory.

The third target is where several computers (probably each with
multiple cores) are to work on a single problem simultaneously. 
Here the objective is again to subdivide the
overall task into many subtasks that can run concurrently on
different computers with access to the same central disk.

%
%
%
%

\section{Preliminaries}

\subsection{Computational Model}
\label{sec:model}

Our  general  approach  to  parallel  computing  is  to
decompose  the  work to  be  done  into  relatively small  units  with
explicit dependencies.  Units all  of whose  dependencies are  met are
called ``runnable''  and a fixed  pool of  worker threads carry  out the
runnable units,  thereby discharging  the dependencies of  other units
and making them  in turn runnable.
A module,  called the \emph{scheduler},  is charged with keeping track of the
dependencies and  scheduling the
execution  of the  individual  units.
Data describing the  units of work
and their input and output data reside  in a shared data store, but we
take considerable care  to ensure that the speed of  this store is not
critical to  overall performance. It can  thus be the large  amount of
shared DRAM on  a multicore server, or disk storage  local to a single
server (for the case where we only have one server, but data too large
for its RAM)  or shared disk provided  sufficiently strong consistency
can be guaranteed.

In our implementations, we have used  two variations of this model. In
one, the \emph{task-model},   the units  of work  are \emph{tasks} and are
relatively
coarse-grained. A task represents, roughly  speaking, all  the work  that
needs to be  done in a particular  part of the matrix  at a particular
stage of  the algorithm. Dependencies  are between tasks, so  one task
cannot execute  until certain others  have completed and it  will find
the data it  needs where those previous tasks have  stored it.
To simplify understanding, we collect different data into
\emph{data packages}, the input and output of the tasks.
For example a typical output of the task $\CC$ is the
data package
$\A = (A, M, K, \rho', E, \lambda)$  with six
components which we refer to as  $\A.{\rm A} $, $\A.{\rm M}$, etc. 

In  the other  model,  the  \emph{job-model}, the  units  of work  are
the \emph{jobs} and are
significantly  more fine-grained  and  represent  a single  elementary
computation such as a  submatrix multiplication. More importantly, the
dependencies  are between  the  jobs  and the  data  they produce  and
consume.   Each job  requires  zero  or more  input  data objects  and
produces one or more output objects. A job is runnable when all of its
input data has  been produced.  This finer grain  approach allows more
concurrency.
Further  gain in efficiency can be achieved by giving the
scheduler guidance as to which jobs or tasks are urgent and which are less
so. This aspect  is mainly ignored in this paper.
The implementation of Meataxe64 \cite{meataxe64} uses the job-model and 
identifies the urgent jobs. 
An implementation in HPC-GAP by Jendrik Brachter and Sergio Siccha is
based on  the task-model.

The parallel  Gaussian elimination  algorithm is described  in Section
\ref{sec:theory}  as a  program called  {\sc the  Chief} charged  with
defining  the tasks  and the  data packages  they work  on.  {\sc  The
  Chief} is described as a sequential program but the reader should be
warned that the tasks are executed in an unpredictable order which may
bear little resemblance to the order  in which the tasks are submitted
by {\sc the Chief}.   The result of running {\sc the  Chief} is a plan
consisting of a list of  tasks, respectively jobs, together with their
inputs and  outputs whose  collective execution performs  the Gaussian
elimination.

Below we  specify {\sc The Chief} in terms of tasks, specified in
turn by jobs, for which the reader will find a more or less specific 
description in Section \ref{jobsandtasks}. 

\subsection{Gaussian elimination}

This subsection describes the Gaussian elimination process in a way we hope is
familiar to the reader, but in our  notation.
Given a matrix $H$ with entries in a field $\F$ the output 
of
the Gauss algorithm consists of matrices $M$, $K$ and $R$ and permutation
matrices $P_{\rho}$ and $P_{\gamma }$ such that
\begin{equation}\label{eq:gauss}
\begin{pmatrix} M & 0\\ K & 1 
\end{pmatrix}P_{\rho} H P_{\gamma} = 
\begin{pmatrix} -1 & R\\ 0 & 0 
\end{pmatrix}.
\end{equation}
The matrices $P_{\rho}$ and $P_{\gamma}$ perform row, respectively column,
permutations on the input matrix $H$ such that the top left-hand part
of the resulting matrix $P_{\rho}HP_{\gamma}$ is invertible (with inverse $-M$)
with the same rank as $H$. 
It should be noticed that the  permutation matrices $P_{\rho}$ and $P_{\gamma}$ in
Equation~(\ref{eq:gauss}) are not uniquely  defined. All that matters is
that they  put the  pivotal rows  and columns  into the  top left-hand
corner of the matrix.  We therefore choose to only specify
the sets of  row and column numbers containing  pivotal elements.
As we are chopping our matrix into blocks, we use the word
\emph{selected} to specify a row or column in which
a pivot has been already been found.
Hence  we  have to apply a permutation  to the rows during  the course of
the algorithm to ensure that our pivots remain located in columns with
increasing indices. We formalize how we store these permutation
matrices in the Definition~\ref{rs}.

\begin{definition}\label{rs}
  When we enumerate elements of a set, we always implicitly
  assume that these are in order. 
	To a subset $\rho = \{ \rho _1,\ldots , \rho _{|\rho |}\}  \subseteq \{ 1,\ldots,\alpha  \} $
        associate       a  $0/1$ matrix 
	$$\rho  \in \F^{|\rho| \times \alpha}  
	\mbox{ with }
\rho _{i,j} = \begin{cases} 1 & j = \rho _i \\ 0 & \mbox{ else.} \end{cases}  $$

	We call $\rho $ the {\em row-select} matrix and
        $\overline{\rho }$ the {\em row-nonselect} matrix
	associated to the set $\rho $ and its complement $\overline{\rho } = \{ 1,\ldots , \alpha \} \setminus \rho $. 
\end{definition}

\begin{remark}\label{rem:perm}
Note that the matrix
$$P_\rho =\left( \begin{array}{c} \rho \\ \overline{\rho } \end{array} \right)
\in \F^{\alpha\times\alpha}$$ is a permutation matrix associated to
the permutation $p_\rho$ with $p_\rho(i) = \rho_i$ for $i \le |\rho|$
and $p_\rho(i) = \overline{\rho}_{i-|\rho|}$ for $i > |\rho|.$ 
Note that $p_{\rho } = 1$ if and only if $\rho = \{ 1,\ldots , i \} $ for 
some $0\leq i \leq \alpha $ but for the other subsets $\rho $
we may recover $\rho $ from $p_{\rho }$ as the image 
$\{ p_{\rho }(1),\ldots , p_{\rho }(i) \}$ if $p_{\rho }(i+1) < p_{\rho }(i)$.
In this
sense we keep  switching
between permutations, subsets, and bitstrings in $\{0,1\}^\alpha$ which
represent the characteristic function of the subset. In particular,
this means that a  subset of cardinality
$2^\alpha - \alpha$ of \emph{special permutations}
of all the $\alpha!$ permutations of $\{1,\ldots,\alpha\}$
is sufficient for all our purposes.
\end{remark}

We  also note  that for  a matrix $H\in \F^{\alpha  \times \beta}  $ and  $\rho
\subseteq \{ 1,\ldots , \alpha \}$ and $\gamma \subseteq \{ 1,\ldots ,
\beta \} $ the matrix $\rho \times  H \in \F^{|\rho|\times \beta } $ consists
of  those rows  of  $H$ whose  indices lie  in  $\rho$ (retaining  the
ordering) and the matrix $H \times \gamma^{tr}  \in \F ^{\alpha \times |\gamma |}$
consists of those columns of $H$ whose indices lie in $\gamma $.
Therefore we also call $\gamma ^{tr} $ the {\em column-select} matrix
and $\overline{\gamma } ^{tr} $ the {\em column-nonselect} matrix associated to $\gamma $.

        \begin{remark} \label{echelon}
		Let $H\in \F^{\alpha \times \beta } $ be of rank $r$.
		Then the echelonisation algorithm will produce  sets
		$\rho \subseteq \{1,\ldots , \alpha \}$ and $\gamma \subseteq \{ 1,\ldots , \beta\}$ of
		cardinality $|\gamma |=|\rho | = r$ and matrices
		$M\in \F^{r\times r}$, $K\in \F^{(\alpha -r) \times r }$, $R\in \F^{r\times (\beta-r) }$
		such that
	$$ \left( \begin{array}{cc} M & 0 \\ K & 1 \end{array} \right) 
	\left( \begin{array}{c} \rho \\ \overline{\rho } \end{array} \right)  H \left( \begin{array}{cc} \gamma^{tr} & \overline{\gamma }^{tr} \end{array} \right)  = 
	\left( \begin{array}{cc} -1 & R \\ 0 & 0 \end{array} \right)  $$
	We will refer to these matrices as
	                $$ (M,K,R,\rho,\gamma ) :=\ech(H) .$$
\end{remark}

Strictly speaking, the job $\ech $ computes the negative 
row reduced echelon form of the input matrix $H$. However we will simply
call this the {\em echelon form of $H$}.

We assume we have an implementation of $\ech$
suitable for use on a single core.
The goal of this paper is to show how to scale it up.

\subsection*{Example}

We now present an example to
highlight some of the structure of the algorithm.
Consider the matrix $C \in \F_3^{6\times 6}$ given by
\[\left( \begin{array}{ccc|ccc}
  0 & 2 & 2 & 0 & 1 & 0\\
  0 & 2 & 2 & 1 & 2 & 2\\
  1 & 0 & 1 & 0 & 2 & 1\\
  \hline
  2 & 0 & 1 & 0 & 2 & 2\\
  0 & 1 & 1 & 2 & 1 & 1\\
  1 & 2 & 2 & 2 & 0 & 0
\end{array}\right)
\]
and divided into four blocks as shown.
Our first step is to echelonise the top left block. This yields
\[
\left( \begin{array}{cc|c}
  0 & 2 & 0 \\
    1 & 0 & 0\\
    \hline
    2 & 0 & 1
  \end{array}\right)
 P_{\{1,3\}}
 \left(
  \begin{array}{ccc}
  0& 2 & 2\\
  0 & 2 & 2 \\
  1 & 0 & 1 
  \end{array}\right)
  P_{\{1,2\}} = 
\left(
  \begin{array}{cc|c} 2 & 0 & 2\\
    0 & 2 & 2\\
    \hline
   0 & 0 & 0 
  \end{array}\right)
\]
matching Equation~\ref{eq:gauss}. Here
\[
 P_{\{1,3\}} =
 \left(
  \begin{array}{ccc}
  1& 0 & 0\\
  0 & 0 & 1 \\
  \hline
  0 & 1 & 0 
  \end{array}\right) \mbox{ and } 
 P_{\{1,2\}} =
 \left(
  \begin{array}{cc|c}
  1& 0 & 0\\
  0 & 1 & 0 \\
  0 & 0 & 1 
  \end{array}\right),\]
  where the bars separate selected from non-selected rows or columns.
So the outputs from this step are
the multiplier $M = \begin{pmatrix} 0 & 2 \\ 1 & 0 
\end{pmatrix}$, $K = \begin{pmatrix} 2 & 0 
\end{pmatrix}$, and $R = \begin{pmatrix} 2 \\ 2
\end{pmatrix}$ as well as the row-select set $\rho$, selecting the
first and third row and  the column-select set $\gamma$, selecting the
first two columns.

After this, two further steps are available, our parallel
implementation will do both concurrently.  We  mimic the row
transformations, applied to the top left block, on the top right
block.
We also use the echelonised top left block to clean out some columns in
the block beneath it. 
We explain them in that order. 

Mathematically, we need to left multiply the top right block by 
$\left( \begin{array}{cc|c}
  0 & 2 & 0 \\
    1 & 0 & 0\\
    \hline
    2 & 0 & 1
  \end{array}\right)
 P_{\{1,3\}}$. We take advantage of the known structure of this matrix
 to speed up this computation as follows.  We divide the top right
 block into the selected rows, as described by $\rho$:
$ \begin{pmatrix}
   0 & 1 & 0\\
   0 & 2 & 1
 \end{pmatrix}$ and   the rest
$ \begin{pmatrix}
1 & 2 & 2
 \end{pmatrix}$. We add the product of $K$ and the selected rows to
 the
 non-selected rows, giving
$ \begin{pmatrix}
1 & 1 & 2
 \end{pmatrix}$, and then  multiply the selected rows by $M$ giving
$ \begin{pmatrix}
   0 & 1 & 2\\
   0 & 1 & 0
 \end{pmatrix}.$ 

 The other step requires us to add multiples of the pivot rows from
 the echelonised top left block 
$ \left( \begin{array}{cc|c} 2 & 0 & 2\\
    0 & 2 & 2\\
 \end{array}\right)$
 to the bottom left block,  so as to
 clear the pivotal columns.
We again take advantage of the known structure of this matrix
to speed up this computation as follows.  We divide the bottom left
block 
into the selected columns, as described by $\gamma$:
$
\begin{pmatrix}
  2 & 0 \\
  0 & 1 \\
  1 & 2
\end{pmatrix}$ and the non-selected columns
$\begin{pmatrix}
1\\1\\2
\end{pmatrix}.$
Now we add the product of the selected columns and $R$ to these
non-selected columns and obtain
$\begin{pmatrix}
2 \\ 0 \\2
\end{pmatrix}.$

We must now mimic the row transformations used to clear the pivotal
columns of the bottom left block by adding multiples of rows
from the top right hand block to the bottom right hand block, i.e.\
we add the product of
$
\begin{pmatrix}
  2 & 0 \\
  0 & 1 \\
  1 & 2
\end{pmatrix}$ and 
$ \begin{pmatrix}
   0 & 1 & 2\\
   0 & 1 & 0
\end{pmatrix}$  to the bottom right hand block, which becomes
$ \begin{pmatrix}
  0& 1 & 0\\
  2 & 2 & 1 \\
  2 & 0 & 2 
\end{pmatrix}.
$

After all these steps, the overall matrix is
\[\left( \begin{array}{ccc|ccc}
  2 &0  & 2 & 0 & 1 & 2\\
  0 & 2 & 2 & 0 & 1 & 0\\
  0 & 0 & 0 & 1 & 1 & 2\\
  \hline
  0 & 0 & 2 & 0 & 1 & 0\\
  0 & 0 & 0 & 2 & 2 & 1\\
  0 & 0 & 2 & 2 & 0 & 2
\end{array}\right).
\]
At this stage we have dealt with all the consequences of the pivots
found in the top left block. What remains to be done is to echelonise
the top-right and bottom-left submatrices in the picture above and deal with
the consequences of any pivots found. Finally part of
the bottom right hand block will need to be echelonised.
This example, chopped $2\times 2$ does not demonstrate pivotal row
merging as described in Section~\ref{subdiv}

\subsection{Some guiding points}

\subsubsection{Subdividing a huge matrix}\label{subdiv}

Our parallel  version of the  Gaussian elimination algorithm  takes as
input  a  huge  matrix  and   subdivides  it,  both  horizontally  and
vertically, into  blocks. We do this  partly to 
obtain concurrency and partly to reduce the row length for cache
reasons.  Once the top-left block has been echelonised, the
same row  operations can be  applied simultaneously to all  the blocks
along the top row.
Putting the  rest of  the left-most block  column into  echelon form
requires  addressing  the  blocks  sequentially  from  the  top  down.
However, in the common case where the co-rank of the top-left block is
small, it is not  a great deal of computation. Once  the top block row
and left-most block  column are in echelon form, we  can update each of
the  blocks  in  the  other  rows and  columns  concurrently.


It should  be remarked that nothing  else can happen until  this first
block  is  echelonised, suggesting  that  we  should make  this  block
smaller  than  the  rest.   A  similar comment  applies  to  the  last
(bottom-right)  block.


Proceeding down the left-most block  column sequentially enables us to
merge into  a single  row of  blocks all those  rows whose  pivots lie
there. This  merging is done to  reduce the amount of  data access.
Usually, the work of  doing this merging is not great, so  that soon after the
first block echelonisation is complete, a large block multiply-and-add
can be done to every block of the matrix.
Without the merging, this multiply-and-add would be done piecemeal,
requiring multiple passes through the block.

\subsubsection{Echelonisation of a single block}\label{sub:echmult}

The performance of the 
echelonisation of a single block (as defined in Remark~\ref{echelon})
can  have a considerable impact  on the
concurrency, as many later jobs may depend on each
echelonisation. 

The  sequential  algorithm used  to  echelonise  individual blocks  is
recursive,  combining elements  of the  recursive inversion  algorithm
already  mentioned,  with  the  greater generality  of  the  technique
described  in  this  paper.   A  block is  divided  into  two,  either
horizontally  or  vertically,  and  the  top  (resp.   left)  part  is
echelonised. Using  a simplified version  of $\CC$ (resp.  $\UR$), see
Sections~\ref{sec:cleardown} and~\ref{sec:updaterow}, the remainder of
the  matrix and  the transformation  matrix are  updated, producing  a
second block  which must also be  echelonised. The results of  the two
echelonisations can be combined to compute the data package consisting
of  $M$,  $K$, $R$,  $P_\rho$  and  $P_\gamma$. Using  this  technique
recursively for all  matrices bigger than a threshold  size (about 300
dimensions) and  a simple direct Gaussian  elimination algorithm below
this size, echelonisation  of a block takes essentially  the same time
as a block multiply.


\subsubsection{Blocks change size and shape}

In    our    description    of   the    algorithm,    especially    in
Equations~\eqref{remnant} and \eqref{DDD}, we imagine that rows are
permuted so 
that those containing pivots (in blocks to the left) are moved to the top,
and the rest are moved to the  bottom.  In the program, however, these two
sets of rows  are held in different matrices, but  it seemed better to
try  to include  all the  information  in one  matrix, attempting  to
clarify the relationships between the parts.

As a  consequence of moving  rows about, the  sizes and shapes  of the
blocks  change  during  the  course   of  the  algorithm.   We  use  a
superscript to indicate  the ``stage'' of a particular  block, so that,
for example, ${\bf C}^j_{ik}$ is the matrix block at the $(i,k)$-position in
its $j$'th stage.  In some ways the original input matrix ${\mathcal
  C}$  (see Equation~\eqref{eq:echelon}  below)  is gradually
converted into  the matrix  ${\mathcal R}$.   An intermediate  matrix ${\bf  B}$
collects the pivotal rows from  ${\mathcal C}$ which are subsequently
deleted from ${\mathcal C}$.

\subsubsection{Riffles}

Our subset-stored permutations are used both to pull matrices apart
and to merge them together. We call the pulling apart an
`extract' where one matrix is separated into two matrices with the
selected rows (or columns) going to one, and the remaining rows to the
other.
We call the merging a `riffle' where one matrix is
assembled from two inputs with the subset-stored permutation
directing from which input each row (or column) comes. 


\subsubsection{Transformation matrix abbreviation}

To compute the transformation matrix, we could apply the same row
operations we performed on the input matrix to the identity matrix. 
In  practice this would be wasteful,
both of memory and  computational effort,  since initially  the identity
matrix contains  no information at  all, and the early  row operations
would mainly be adding zeros  to other zeros.  Considerable effort has
been expended  in this paper  to avoid  storing any parts  of matrices
whose values are  known \emph{a priori},
thereby saving both  memory, and work manipulating them. The graphs
shown in Section~\ref{sec:exper} suggest that this has been successful. The
details of this optimisation are the subject of Section~\ref{sec:urt}, which
may be skipped on first reading.


Note that  although the  output matrix ${\mathcal  M}$ in
Equation~\ref{eq:echelon} below
is  square, the
number of blocks in  each row may differ from the  number of blocks in
each column, since the block rows are indexed by the block column in which
the pivot was found and \emph{vice versa}.

\section{A parallel Gauss algorithm}  
\label{sec:theory}

\subsection{The structure of the algorithm}

We  now  describe  a  parallel version  of  the  Gaussian  elimination
algorithm which takes as input a huge matrix and subdivides it, both
horizontally and vertically, into blocks.

To distinguish between the
huge matrix and its blocks we use different fonts.

Let $\F$ be a finite field
and ${\mathcal C} \in \F^{m \times n }$ 
a huge matrix of rank $r$. 
We describe a parallel version
of the well-known {\sc Gauss} algorithm, which computes matrices
${\mathcal R}$, a transformation matrix ${\mathcal T}= \begin{pmatrix} \TT & 0_{r\times (m-r)} \\ \KK & 1_{(m-r) \times (m-r) } \end{pmatrix}$ and subsets $\varrho
\subseteq \{1,\ldots m\}$ and $\Upsilon \subseteq \{1,\ldots, n \}$,
both of cardinality $r$,
such that 
\begin{equation}\label{eq:echelon}
\begin{pmatrix} \TT & 0_{r\times (m-r)} \\ \KK & 1_{(m-r) \times (m-r) } \end{pmatrix}
\begin{pmatrix} \varrho \\ \overline{\varrho} \end{pmatrix} {\mathcal C} 
	\begin{pmatrix} \Upsilon^{tr} &  \overline{\Upsilon}^{tr} \end{pmatrix} 
= \begin{pmatrix} -1_{r\times r} & {\mathcal R} \\ 0 & 0 \end{pmatrix} 
\end{equation}
is in (negative row reduced) echelon form.

Comparing to  Remark~\ref{echelon}, we see  that our  task is to  chop our
huge input  matrix into  smaller blocks and  then, using  ${\ech}$ and
other jobs  on the blocks, to  effect the same operation  on the huge
matrix  as ${\ech}$  does  on  a single  block.   We therefore  choose
positive integers $a_i,b_j$ such that
$$\sum_{i=1}^a a_i = m , \ \sum _{j=1}^b b_j = n $$ 
and our algorithm copies the block-submatrices of 
the input matrix ${\mathcal C}$ into  data packages
$\C_{ij}\in  \F^{a_i \times b_j}$, called
\emph{blocks}, 
and performs tasks (as described in Section \ref{jobsandtasks})
on these smaller blocks.

We call the $a$ matrices $\C_{1j},\ldots, \C_{aj}$ the 
$j$-th \emph{block column}
and the $b$ matrices $\C_{i1},\ldots, \C_{ib}$ the $i$-th \emph{block
  row} of ${\mathcal C}$.
  Echelonising the overall matrix ${\mathcal C}$ is achieved by
  performing an echelonisation 
algorithm on 
individual blocks and using the resulting data to modify others.

The result of the Gaussian elimination as well as the intermediate  matrices
are partitioned into blocks: 
When the Gauss algorithm has completed, the matrix ${\mathcal R}$, which occurs in
Equation~(\ref{eq:echelon}),  consists of blocks and  has the form
\begin{equation} \label{remnant} 
	{\mathcal R}  = \begin{pmatrix} 
	R_1 & R'_{12} & \ldots & \ldots & R'_{1b} \\ 
	0 & R_2 & R'_{23} & \ldots & R'_{2b}  \\
\vdots & \ddots & \ddots & \ddots &  \vdots  \\
	0  & \ldots &  0 &  R_{b-1} &  R'_{b-1,b} \\
0  & \ldots & \ldots  & 0 &  R_{b} 
\end{pmatrix}
\end{equation}
with $R_j\in \F^{r_j\times (b_j-r_j)}$ and $R'_{jk} \in \F ^{r_j \times (b_k-r_k)} $.
Here $r=\sum_{j=1}^b r_j$ is the rank of ${\mathcal C}$ and for $k=1,\ldots ,b$ the sum
$\sum_{j=1}^k r_j $  is the rank of the submatrix of ${\mathcal C}$ consisting of 
the first $k$ block columns.

The time-consuming parts of the Gaussian elimination algorithm consist of
Step~1 and Step~3, whereas the  intermediate Step~2 is not.
After the first step the matrix ${\mathcal C}$ has been transformed
into an upper 
triangular matrix and prior to permuting columns the matrix
$(-1_{r\times r}|{\mathcal R}) \in \F ^{r\times n }$  has the shape 
\begin{equation} \label{DDD} 
	\tilde{{\mathcal R}} =
 \begin{pmatrix}
-1 \mid R_1 & X_{12} \mid R_{12} & \ldots & \ldots & X_{1b} \mid R_{1b} \\
0 \mid 0 & -1 \mid R_2 & X_{23} \mid R_{23} & \ldots &  X_{2b} \mid R_{2b} \\
\vdots & \ddots & \ddots & \ddots &  \vdots  \\
	0 \mid 0 & \ldots &  0 \mid 0 & -1 \mid R_{b-1} & X_{b-1,b} \mid R_{b-1,b} \\
0 \mid 0 & \ldots & \ldots  & 0 \mid 0 & -1 \mid R_{b} \\
\end{pmatrix},
\end{equation}
where $R_{jk} \in \F^{r_j\times (b_k-r_k)}$
and $X_{jk} \in \F ^{r_j\times r_k}$. 
To simplify notation in the  algorithms  below we define the 
data packages 
$\B_{jk} = (X_{jk} | R_{jk}) \in \F ^{r_j\times b_k} $
for $1\leq j<k\leq b$  and $\D_j=(\D_j.{\rm R},\D_j.\gamma )$, 
where $\D_j.{\rm R} = R_j$ and $\D_j.\gamma \subseteq \{1,\ldots , b_j\}$ is the set of indices of
pivotal columns in the block column $j$.

\subsubsection{Computing the transformation matrix}
During the course of the algorithm we also compute the
transformation matrix ${\mathcal T} \in \GL_{m}(\F )$
as given in Equation~(\ref{eq:echelon}). Of course this could be
achieved by simply performing the same row operations on an identity
matrix that were performed on ${\mathcal C}.$ This involves considerable
work on blocks known to be either zero or identity matrices. For example, if
${\mathcal C}$ is invertible, the work required to compute its inverse
is needlessly increased by 50\%. To avoid such extra work,
we only store the relevant parts of the blocks of the transformation matrix.

During    the   computation,    the   data    packages   $\M_{ji}    $
($j=1,\ldots,b;i=1,\ldots,a$) and $\K_{ih}$ ($i,h=1,\ldots ,a$) record
the matching status of the transformation matrix.  If $\tilde{C}_{ik}$
denotes the $i,k$-block  of the original input  matrix ${\mathcal C}$,
then initially $\C_{ik} = \tilde{C}_{ik}$  and $\B_{jk} $ has no rows.
Likewise, initially $\K_{ih} $ is the identity matrix if $i=h$ and the
zero matrix otherwise and the matrix $\M_{ji} $ has no rows and no columns. When storing
the data packages $\K $ and $\M $, we omit the columns and rows in the
blocks that are zero or still unchanged from the identity matrix.

To obtain the row select matrix we maintain further data packages 
$\E_{ij} = (\E_{ij}.\rho ,\E_{ij}.\delta )$ ($1\leq i\leq a, 1\leq j\leq b$)
with $\E_{ij}.\rho \subseteq \{1,\ldots , a_i \}$ is the set of indices of 
pivotal rows in block row $i$ with pivot in some block column $1,\ldots ,j$
and $\E_{ij}.\delta \in \{0,1\}^{|\E_{ij}.\rho|}$ records which indices already
occurred up to block column $j-1$.

During the algorithm, having handled block row $i$ 
in Step~1, we have 
$$\sum _{h=1}^a \M_{jh} \times \E_{hi}.\rho \times \tilde{C}_{hk} = \B_{jk} $$
and 
$$\sum _{h=1}^a \K_{\ell h} \times \overline{\E_{hi}.\rho} \times \tilde{C}_{hk} = \C_{\ell k} .$$
The final
column select matrix is the block diagonal matrix 
$$\Upsilon = \diag (\D_1^{a}.\gamma ,\ldots , \D_b^{a}.\gamma ).$$
and the row select matrix is 
$$\varrho = \diag (\E_{1b}.\rho ,\ldots , \E_{ab}.\rho )  .$$

\subsection{Step~1} \label{sec:step1}

Step~1 loops over the block rows. 
For the $j$-th column
of the $i$-th block row, the algorithm calls Task $\CC$ with the two
data packages $\C_{ij}$ and $\D_j^{i-1}$ as input. 
Task $\CC$ amalgamates the pivots in $\D_j^{i-1}.\gamma $
with  the pivots in the 
matrix $\C_{ij}$ to produce the
enlarged set $\D_j^{i}.\gamma $ as well as 
a new matrix $\D_j^i.{\rm R}$ of 
the (negative)  echelon form $(-1\mid  \D_j^i.{\rm R})$
followed by $0$ rows (up to column permutations 
which are remembered in $\D_j^i.\gamma $). 
Moreover, the task  
$\CC$ records in its output data package $\A _{ij} $ 
the row operations performed. 
 With the help of these data packages,
the first step then propagates the same elementary row operations
to the remaining blocks in block row $i$
as well as to block row $i$ of the transition matrix 
 using Task $\UR$.

Hence Step~1 assembles in the block row
$\begin{pmatrix} 0 &\ldots& 0& -1\!\mid\! \D_j.{\rm R}&  \B_{j,j+1}&  \ldots&
\B_{jb}
\end{pmatrix}
$
 the  rows of the original input matrix
whose pivotal entries  lie in block
column $j$ for $j \leq b$.  
These rows are then deleted from the data package $\C$. 
Thus having treated the $j$-th block column
the matrix $\C$ contains no rows whose pivots lie in block columns
$1,\ldots, j $ and the number of rows  of $\C$ is $m-\sum_{k=1}^j r_k$.

In particular, during the course of the entire algorithm 
the block rows $\C_{i,-}$  contain fewer and fewer rows, whereas the 
number of rows of the block row $\B_{j,-}$ increases accordingly. 
After completing the block column $j$ the matrices $\B_{jk}$ remain stable.

Similarly for  the transformation matrix,  the matrix $\M$ gains rows
whereas $\K$ loses rows.
However, things here  are slightly more
complicated  due  to   the  fact  that  we  do  not   store  the  full
transformation matrix.  
As we only store columns of $\K_{ih}$ that are not known a priori to be
zero or columns of an identity matrix, we have to ensure that all the needed
columns are present when calling $\URT$. The columns that are not yet
present are precisely the columns that correspond to the positions
of the  pivot rows and are stored in $\E_{hj}.\delta $.
 If $i=h$ this means we need
to insert into the correct positions columns of an identity matrix
and if $i\not=h$ then columns of a zero matrix. 
To achieve this, and also to efficiently deal with the cases where the matrices $\K _{ih}$ 
or $\M_{jh}$ are to be initialized we adapted the task $\UR $ to obtain 
$\URT$.

\subsection{Step~2}
This intermediate step becomes necessary as we do only store the relevant parts of 
the transformation matrices $\M _{ji}$. Before the upwards cleaning in Step~3 we
need to riffle 
in zero columns in $\M _{ji}$ so that the number of columns 
in $\M _{ji}$ is equal to the number of columns in $\M _{bi} $ for all $j$. 

\subsection{Step~3}

Then back cleaning only performs upwards row operations on the matrix
from Equation \eqref{DDD}  
to eliminate the $X_{jk}$. 
The matrices $R_{jk} $  from Equation \eqref{DDD} are stored in the data packages ${\bf R}_{jk}$ in 
{\sc the Chief}. 
Having cleaned block columns $b,\ldots , k-1$ the algorithm adds the $X_{jk}$ multiple of block row $k$ 
to block row $j$ for all $j\leq k-1$ to clear block column $k$. 
The same row operations are performed on the 
relevant part $\M $ of the transformation matrix.
\\

\begin{algorithm}[H]  
\caption  {{\sc The Chief}}
\SetKwData{Left}{left}\SetKwData{This}{this}\SetKwData{Up}{up}
\SetKwFunction{Union}{Union}\SetKwFunction{FindCompress}{FindCompress}
\SetKwInOut{Input}{Input}\SetKwInOut{Output}{Output}
\SetKwInOut{Exput}{Exput}

\Input{${\mathcal C}=:(\C^{1}_{ik})_{i=1,\ldots , a, k=1,\ldots , b}$, where $\C^1_{ik} \in \F^{a_i\times b_k}$}
\Output{${\mathcal R}=(\R^{k-j}_{jk}) _{j\leq k=1,\ldots,b}$, $\TT = (\M^{2a}_{jh})_{j=1,\ldots,b,h=1,\ldots, a}$, $\KK = (\K^{a}_{ih})_{i,h=1,\ldots , a}$, 
a row select matrix $\varrho \subseteq \{ 1,\ldots ,m\} $, 
	the concatenation of the  $\E_{ib}.\rho\subseteq \{ 1,\ldots , a_i\}$ ($i=1,\ldots,a$), 
	and
	a column select matrix $\Upsilon \subseteq \{1,\ldots , n\} $, the concatenation of the  $\D^{a}_{j}.\gamma \subseteq \{ 1,\ldots , b_j\} $ ($j=1,\ldots ,b$),  such that 
$$ \begin{pmatrix} \TT & 0 \\ \KK & 1 \end{pmatrix} 
\begin{pmatrix} \varrho  \\ \overline{\varrho}  \end{pmatrix} {\mathcal C}  
\begin{pmatrix} \Upsilon &\overline{\Upsilon } \end{pmatrix}  =    \begin{pmatrix} -1 & {\mathcal R} \\ 0 & 0 \end{pmatrix} $$}
 \BlankLine
\hspace*{-1cm}{\sc Step 1}: \\
	   \For{ $i$ from 1 to $a$}{
	\For{ $j$ from 1 to $b$}{
		$(\D^{i}_j; \A_{ij}) := \CC(\C^{j}_{ij},\D^{i-1}_{j},i) $\;  
		$\E_{ij} := ${\sc Extend}$(\A_{ij},\E_{i,j-1},j) $\;
	   \For{ $k$ from $j+1$ to $b$}{ 
	   $(\C^{j+1}_{ik},\B^{i}_{jk}) := \UR (\A_{ij}, \C^{j}_{ik},\B^{i-1}_{jk},i)$\;} 
	   \For{ $h$ from $1$ to $i$}{
		   $(\K^{j+1}_{ih},\M^{i}_{jh}) := ${\sc UpdateRowTrafo}$(\A_{ij},\K^{j}_{ih},\M^{i-1}_{jh}, \E_{hj},i,h,j)$\;
        }
}
}
           \hspace*{-1cm}
	           {\sc Step 2:}  \\
	 \For{$j$ from 1 to $b$}{ 
		 \For{$h$ from 1 to a}{ 
			 $\M^{a+1}_{jh} := ${\sc RowLengthen}$(\M^{a}_{jh} , \E_{hj}, \E_{hb} )$\; 
	 }
        }
           \hspace*{-1cm}
	{\sc Step 3:}  \\
	   \For{$k$ from $1$ to $b$}{
	   $\R^{0}_{kk} := ${\sc Copy}$ (\D^{a}_k)$\;}
	   \For{$k$ from $b$ downto $1$}{
	 \For{$j$ from 1 to $k-1$ }{
		 $(\X_{jk},\R^{0}_{jk}) := \PCU(\B^{a}_{jk},\D ^{a}_k) $\;
	\For{ $\ell $ from k to b }{
		$\R^{\ell-k+1}_{j\ell}:=\CU (\R^{\ell-k}_{j\ell},\X_{jk},\R^{\ell-k}_{k\ell}) $\;
        }
	\For{ $h$ from 1 to a }{
	$\M^{a+h}_{jh}:=\CU(\M^{a+h-1}_{jh},\X_{jk},\M^{a+h-1}_{kh} )$\;} 
        	}
	}
\end{algorithm}

\section{Jobs and Tasks}\label{jobsandtasks}

\SetAlgorithmName{Task}{Task}{task} 
\addtocounter{algocf}{-1}

\subsection{The jobs}


In this section we describe the jobs. These are fundamental steps
that are  later used  to define the  tasks. Many of  the jobs  take as
input one or more matrices.
While the input and output matrices of the
jobs within  the global  context of the  parallel Gauss  algorithm are
blocks computed from  a huge input matrix, the jobs  described in this
section work locally only on these matrices. In current implementations,
each job can be performed by a single threaded computation, entirely in
RAM and in a reasonable amount of time.

\begin{description}
\item [{\sc cpy}] This task simply copies the input matrix to the output.
\item[{\sc mul}]  This job performs a matrix {\bf mul}tiplication. It  
takes as input two matrices $A\in \F^{\alpha\times\beta}$ and $B\in
\F^{\beta\times \delta}$
and returns as output the matrix $A\times B\in\F^{\alpha\times \delta}$.  
\item[{\sc mad}]  This job performs a matrix {\bf m}ultiplication followed
by a matrix {\bf ad}dition.  It
takes as input  matrices $A\in \F^{\alpha\times\delta}$ and $B\in
\F^{\alpha\times \beta}$ and $C\in\F^{\beta\times\delta}$ and 
returns the matrix $A + B\times C\in
\F^{\alpha\times \delta}.$
\item[{\sc ech}] This job performs an {\bf ech}elonisation as described in 
	Remark \ref{echelon}. 
We will refer to the job as
		$$ (M,K,R,\rho,\gamma ) :=\ech(H) .$$
%
%
%
%
%
\item[{\sc cex}] This job performs two {\bf c}olumn {\bf ex}tracts.
It takes as input a matrix $H\in \F^{\alpha\times\beta}$
and a subset $\gamma\subseteq\{1,\ldots,\beta\}$ and 
returns the matrices $H\times \gamma^{tr}$ and $H\times \overline{\gamma}^{tr}$,
 consisting of all those columns
of $H$ whose indices lie in $\gamma$, respectively do not lie in
$\gamma$,  as described in Definition~\ref{rs}.
\item[{\sc rex}] This job performs two {\bf r}ow {\bf ex}tracts.
It takes as input a matrix $H\in \F^{\alpha\times\beta}$
and a subset $\rho\subseteq\{1,\ldots,\alpha\}$ and 
returns the matrices $\rho \times H$  and $\overline{\rho}\times H$, 
consisting of all those rows
of $H$ whose indices lie in $\rho$, respectively do not lie in
$\rho$,  as described in Definition~\ref{rs}.

\item[{\sc unh}] 
This job performs a {\bf un}ion plus {\bf h}istory. It takes as input  a subset
$\rho_1 \subseteq \{ 1,\ldots , \alpha \}$ and,
for $\alpha_0 = \alpha - |\rho_1|$, 
a subset $\rho_2 \subseteq \{ 1,\ldots , \alpha_0 \}$ and 
returns a subset $\rho\subseteq \{ 1,\ldots , \alpha \}$ defined as
follows. 
Write $\{1,\ldots,\alpha\} \setminus
\rho_1 = \{x_1, \ldots, x_{\alpha_0} \}$. Define
$$\rho=\rho_1 \cup \{ x_i \mid i \in \rho_2 \} =:
\{y_1, \ldots, y_{r} \}$$
as an ordered set. Then  $u\in \{0,1\}^{r} $ with 
 $u_\ell =0$ if $y_\ell \in \rho_1$ and $u_\ell  = 1$ 
otherwise.  We refer to this job as
	   $$(\rho,u) := \PVC (\rho_1,\rho_2 ) .$$

\item [{\sc un0}] This job does the same as {\sc unh} except that the
  first input  set of {\sc unh} is omitted and assumed empty. 
           
\item[{\sc mkr}] It takes two sets $\rho _1 \subseteq \rho _2 $ and 
	produces a bitstring 
$\lambda \subseteq \{0,1\}^{|\rho _2 |}$ with $1$s corresponding to the elements
in $\rho_1$ and $0$s corresponding to the elements in
$\rho _2\backslash \rho_1.$

\item[{\sc rrf}] This job performs a row riffle.
The input consists of a bit string $u\in\{0,1\}^r$ and two matrices
$B\in \F^{\alpha\times\beta}$ and $C\in \F^{\gamma\times \beta}$ with
$\alpha+\gamma=r$, where the number of $0$s in $u$ is $\alpha$ and the
number of $1$s in $u$ is $\gamma$. The job returns the new matrix
$A\in  \F^{r\times b}$ whose rows are the rows of $B$ and $C$
combined according to $u$. In some sense this is  the inverse of row extract.
\item[{\sc crz}] Similarly to row riffles we also need 
	column riffles, but we only need to riffle in zero columns.
\item[{\sc adi}] It takes as input a matrix $K\in \F^{\alpha\times\beta}$ 
	and a bitstring $\delta \in \{ 0,1 \} ^{\beta } $ 
and puts 
	$K_{i,j_i}:= 1$ if $j_i$ is the position of the $i$th $0$ in $\delta $.
	Note that combining {\sc crz} with {\sc adi} allows us to riffle in
	columns of the identity matrix. 
\end{description}

\subsection{The tasks}\label{sec:tasks}

We now describe the tasks on which our Gaussian elimination algorithm 
depends. As mentioned above, a task receives data packages as input,
which in turn may consist of several components, 
and returns data packages as output.

\begin{definition}\label{ready}
	A \emph{data package} is  a record of one or several components. 
	A data package is called \emph{ready} (for a given scheduler) 
	if the task that produces it as 
	output has finished, regardless whether this task has computed all its components.
	If there is no task having this data package as an output, then 
	we also consider it ready.
\end{definition}

{\bf Example}: Task $\UR $. \\
 If  called  with  the
parameter $i=1$  the task  $\UR $ can start,  even though the component  $\A.{\rm A}$ of
the data package $\A$ has not been computed, after the task $\CC $ for $i=1$ 
has completed. 
Note that for $i=1$, no job in the task $\UR $ takes the component $\A.{\rm A} $ as an input.
Also the data package $\B_{jk}^{0} $ is an input to $\UR $ but not computed by any task in
 {\sc the Chief}. So therefore it is also considered ready. 
\\

\begin{algorithm}[H]
\caption{{\sc Extend}\label{extend}}
\SetKwInOut{Input}{Input}\SetKwInOut{Output}{Output}

\Input{$\A=(A,M,K,\rho',E,\lambda )$,  $\E = (\rho,\delta)$ with $\rho \subset \{1,\ldots , \alpha \} $,
$\delta $  a riffle, $j$}
\Output{$\E  $.} 
\BlankLine
\begin{tabular}{lll}
	$j=1$ & (UN0):& $(\E.\rho , \E.\delta ) := \PC0 (\A.\rho' )$;\\
	$j\neq 1$ & (UNH):& $(\E.\rho , \E.\delta ) := \PVC (\E.\rho , \A.\rho' )$;\\
\end{tabular}
\end{algorithm}
\mbox{} \\

%

\begin{algorithm}[H]
\caption  {{\sc RowLengthen}\label{rl}}
\SetKwData{Left}{left}\SetKwData{This}{this}\SetKwData{Up}{up}
\SetKwFunction{Union}{Union}\SetKwFunction{FindCompress}{FindCompress}
\SetKwInOut{Input}{Input}\SetKwInOut{Output}{Output}

\Input{$\M  \in \F^{\alpha \times g_1} $, $\E_1.\rho
  \subseteq \E_2.\rho \subseteq \{1,\ldots , \alpha \}$ of
  sizes $g_1, g_2$ with $g_1 \le g_2$.}
\Output{
$\M \in \F^{\alpha \times g_2} $.} 
\BlankLine
(MKR):\quad $\lambda := \MKR ( \E_1.\rho , \E_2.\rho ) $\;
(CRZ):\quad $\M := \CRZ(\M,\lambda)$\; 
\end{algorithm}
\mbox{} \\

\begin{algorithm}[H]
\caption  {{\sc ClearUp}\label{ClearUp}}
\SetKwData{Left}{left}\SetKwData{This}{this}\SetKwData{Up}{up}
\SetKwFunction{Union}{Union}\SetKwFunction{FindCompress}{FindCompress}
\SetKwInOut{Input}{Input}\SetKwInOut{Output}{Output}

\Input{$\R\in \F^{\alpha \times \beta}$,  $\X \in \F ^{\alpha \times
    \gamma} $, $\M \in \F^{\gamma \times \beta }$.}
\Output{$\R \in \F^{\alpha \times \beta} $.} 
\BlankLine
(MAD): \quad $\R := \R + \X \times \M $\;
\end{algorithm}
\mbox{} \\

\begin{algorithm}[H]
\caption  {{\sc PreClearUp}\label{PreClearUp}}
\SetKwData{Left}{left}\SetKwData{This}{this}\SetKwData{Up}{up}
\SetKwFunction{Union}{Union}\SetKwFunction{FindCompress}{FindCompress}
\SetKwInOut{Input}{Input}\SetKwInOut{Output}{Output}

\Input{$\B\in \F^{\alpha \times \beta}$,  $\D $, with $\D.\gamma
  \subseteq \{ 1,\ldots , \beta \}$ of cardinality $g$.}
\Output{
$\X \in \F^{\alpha \times g} $, 
$\R \in \F^{\alpha \times (\beta - g)} $.} 
\BlankLine
(CEX): \quad 
$\X := \B \times \D.\gamma^{tr} $; $\R := \B \times \overline{\D.\gamma}^{tr} $\;
\end{algorithm}
\mbox{} \\

\begin{algorithm}[H]
\caption  {{\sc Copy}\label{Copy}}
\SetKwData{Left}{left}\SetKwData{This}{this}\SetKwData{Up}{up}
\SetKwFunction{Union}{Union}\SetKwFunction{FindCompress}{FindCompress}
\SetKwInOut{Input}{Input}\SetKwInOut{Output}{Output}

\Input{$\D $, with $\D.{\rm R} \in \F^{\alpha \times \beta} $.}
\Output{
$\R \in \F^{\alpha \times \beta } $.} 
\BlankLine
(CPY): \quad $\R := \D.{\rm R} $\;
\end{algorithm}
\mbox{} \\

\subsubsection{Task $\CC$}\label{sec:cleardown}

Task $\CC$ works on block columns. 
Suppose that $j\in \{1,\ldots,b\}$ and $\CC$ works on block column
 $j$ which contains
$b_j$ columns. Task $\CC$ assumes that 
block column $j$  truncated after row $i-1$ is in row echelon form
and the aim of task $\CC$ is to replace  the block column $j$
truncated
after row $i$ by its row echelon form.

Task $\CC$ takes two
data packages $\C$ and $\D$ as input. The first data package $\C$
is the block
$\C_{ij}$ which is the block in the $i$-th block row of
block column $j$. The second data set $\D$ contains two data elements.
The data element $\D.{\rm R}$ is a matrix such that block column $j$ 
truncated after block row $i-1$ is in row echelon form $(-1\mid \D.{\rm R})$
followed by $0$ rows.
The data element $\D.\gamma
\subseteq  \{ 1,\ldots  ,b_j \}$ 
contains indices of the pivots assembled in block column
$j$ truncated after block row $i-1$. 

The task produces two data packages $\A$  and $\D$ as outputs.
The data elements stored in the data
package $\A$ are required to propagate row operations performed during
the call to Task $\UR$ to other blocks in block row $i$. The data
elements stored in the data package $\D$ are required for a 
subsequent call to $\CC$ for the block $\C_{i+1,j}$ in block column $j$.

We begin by partitioning 
the input block $\C$  according  to
$\D.\gamma$ into pivotal and non pivotal columns $\C = (\A.{\rm A}\mid A')$.  
Using the rows of the matrix $(-1\mid \D.{\rm R})$ we can
reduce $\C$ to $(0 \mid H')$  where $H'=A'+\A.{\rm A}\times \D.{\rm R}$. 
The next step is to call job ${\ech}$  to
echelonise $H'$
and obtain
$$ (\A.{\rm M},\A.{\rm K},R,\A.\rho',\gamma' ) :=\ech(H'), $$
where $\A.\rho'$ is the set of pivotal rows of $H'$ and
$\gamma'$ the set of pivotal columns.

As block column $j$  truncated after block row $i-1$ is in row echelon
form $(-1_{r\times   r} \mid \D.{\rm R})$
followed by $0$ rows, we now wish to determine a new remnant
matrix $\hat{R}$  (which will become the new $\D.{\rm R}$) 
such that  block column $j$ 
truncated after block row $i$ is in row echelon form $(-1_{(r+r')\times
 (r+r')}\mid \hat{R})$ followed by $0$ rows.
To achieve this, the we have to add
the $r'$ pivots of $H'$ to  $-1_{r\times r}$
and reduce $\D.{\rm R}$ according to $(-1_{r'\times r'}\mid R)$.
This  amounts to first separating
the columns of $\D.{\rm R}$ into those containing pivot entries of $H'$ and
those that do not, i.e.\ writing $\D.{\rm R}=(\A.{\rm E} \ \mid \ R')$ with the help of the
row select and row non-select matrices $\gamma'$ and $\overline{\gamma'}$.
We then use the rows of the matrix 
$(-1_{r'\times r'}\mid R)$ to reduce $\D.{\rm R}$ to $(0\mid R' + \A.{\rm E}\times \D.{\rm R})$.
The new set $\D.\gamma$ of all pivotal columns of block column $j$
truncated
after block row $i$ is now obtained by combining the old set $\D.\gamma$ and
$\gamma'$. We record in $\A.\lambda$ the information which of these
indices came from the $r'$ pivots of $H'.$ Finally, the new remnant
$\hat{R}$ is obtained by interleaving the rows of 
$R' + \A.{\rm E}\times R$ with the rows of $R$ according to $\A.\lambda$ and
storing the resulting matrix as the new  $\D.{\rm R}.$ 

The following pseudo code details Task $\CC$:\\
\bigskip
\SetAlgoNoLine
\begin{algorithm}[H]
\caption  {{\sc ClearDown}\label{CC}}
\SetKwData{Left}{left}\SetKwData{This}{this}\SetKwData{Up}{up}
\SetKwFunction{Union}{Union}\SetKwFunction{FindCompress}{FindCompress}
\SetKwInOut{Input}{Input}\SetKwInOut{Output}{Output}

\Input{$\C \in \F^{\alpha \times \beta}$,  $\D.\gamma \subseteq \{ 1,\ldots , \beta\} $ of cardinality $r$, $\D.{\rm R} \in \F ^{r\times (\beta-r) }$, $i$;}
\Output{$\D.{\rm R} \in \F^{(r+r') \times (\beta-r-r') }$,  $\D.\gamma \subseteq \{ 1,\ldots , \beta\}$ of cardinality $r+r'$ and
$\A=(A,M,K,\rho',E, \lambda)$  where $A \in \F^{\alpha \times r}$, $M\in \F^{r' \times r' }$, $E\in \F^{r\times r' } $, $K\in \F^{(\alpha -r')\times r'}$,  $\rho'\subseteq \{ 1,\ldots , \alpha-r \} $
	of cardinality $r'$, $\lambda  \in \{ 0,1\}^{r+r'}$.
}
\BlankLine
\uIf{$i=1$}{
	(ECH):\quad $( \A.{\rm M},\A.{\rm K}, \D.{\rm R},\A.\rho',\D.\gamma) := \ech(\C)  $\;
}
\Else{
(CEX):\quad $\A.{\rm A}:= \C \times \D.\gamma ^{tr}$;\quad $A':=\C \times \overline{\D.\gamma}^{tr} $\;
(MAD):\quad  $H:= A'  + \A.{\rm A} \times \D.{\rm R} $\;
(ECH):\quad $( \A.{\rm M},\A.{\rm K},R,\A.\rho',\gamma') := \ech(H)  $\;
(CEX): \quad  $\A.{\rm E}:=\D.{\rm R} \times (\gamma')^{tr}$, $R':=\D.{\rm R}\times (\overline{\gamma'})^{tr}$\;
(MAD):\quad  $R':=R' + \A.{\rm E} \times R $\; 
(UNH): \quad $(\D.\gamma,\A.\lambda) := \PVC( \D.\gamma , \gamma') $\;
  (RRF): \quad  $\D.{\rm R}:=\RRF(\A.\lambda,R',R)$\;
}
\end{algorithm} 

\subsubsection{Task $\UR$}\label{sec:updaterow}

Given $i\in \{1,\ldots,a\}$, the Task $\UR$ works on block $\C=\C_{ik}$ in 
block row $i$ and block column $k$.
It  takes as input data packages $\A$, $\C$ and $\B$,
where the data package $\A$ encodes the necessary information computed
by $\CC$ when transforming
an earlier  block in the same block row $i$ 
into echelon form.

The same row operations that were  performed on this earlier block now need
to be performed on $\C$.  This subroutine also assembles in the matrix
$\B$ the  rows in block column  $k$ whose pivotal entry  lies in block
column $j$ for $j+1\leq k \leq b$.  The new data package $\C$ returned by
Tasks $\UR$  then is equal  to the transformed input  matrix $\C$ with  these rows
deleted.

%
%

The following pseudo code details Task $\UR$:\\

\begin{algorithm}[H]
\caption  {{\sc UpdateRow}\label{UR}}
\SetKwData{Left}{left}\SetKwData{This}{this}\SetKwData{Up}{up}
\SetKwFunction{Union}{Union}\SetKwFunction{FindCompress}{FindCompress}
\SetKwInOut{Input}{Input}\SetKwInOut{Output}{Output}
\SetKwInOut{Exput}{Exput}

\Input{$\A=(A,M,K,\rho',E,\lambda )$,  $\C \in \F^{\alpha \times \beta}$, 
	 $\B \in \F^{r\times \beta}$, $i$. } 
\Output{$\C\in \F^{(\alpha -r')\times \beta} $, $\B\in \F^{(r+r')\times \beta}$.}
\BlankLine


\begin{tabular}{l|lcl}
	$(1)$ &  $i\neq 1$ &(MAD): & $Z := \C+\A.{\rm A}  \times  \B $;\\
&   $i=1$  & (CPY):  & $Z := \C;$\\
  \hline
  $(2)$&  always &  (REX): & $V:=\A.\rho'  \times  Z$; and $W:=\overline{\A.\rho'}\times Z$;\\
  $(3)$&  always &	(MUL): & $X:=\A.{\rm M} \times  V $;\\
  $(4)$&  $i\neq 1$ & (MAD):&  $S:=\B+\A.{\rm E} \times X$;\\
  \hline
  $(5)$&   $i\neq 1$ & (RRF): &  $\B:=\RRF (\A.\lambda,S,X)$;\\
&   $i=1$ & (CPY): & $\B := X$;\\
  \hline
  $(6)$& always  & (MAD): & $\C := W+\A.{\rm K} \times V $;\\
\end{tabular}
\end{algorithm}

\begin{remark}\label{remread} 
	In the case $i=1$ in Task $\UR $ 
 we work with the first block row. 
 Therefore we do not need to perform the upwards cleaning on the data package $\C $  
 and the data package $\B $ is initialized accordingly.
 Note that for $i=1$ 
 the task $\CC $ did not compute the components 
 $\A.{\rm A}$, $\A.{\rm E}$  and $\A.\lambda $ and also the input data package $\B $ of $\UR $ is not present. 
 \end{remark}

\subsubsection{Task $\URT $}\label{sec:urt}

If one is not too concerned about performance, then it would be
possible to generate an identity matrix $\K$  and apply the $\UR $ task 
replacing $\C$ by $\K$ and $\B$ by $\M$ 
to mimic the relevant row operations performed to obtain the
transformation matrix $\M$ and the cleaner matrix $\K$. This would
result in a lot of needless work 
performed on zero or identity matrices. 
The main difference is that we never store any columns known to belong
to an identity or a zero matrix. Instead we insert these columns just
before they are needed.
Moreover, we never add a matrix known to be
zero or multiply by  a matrix known to be the identity.

As a result, we require a separate procedure, $\URT$, to mimic the row
operations on  the transformation  matrix.  $\URT$ still  performs the
same  steps as  $\UR$, identified  by  the same  numbers, however  it
requires some  additional steps,  indicated by the  symbol $+$  in the
first column  and which insert  some unstored columns into  the matrix
$\K$. The various instances of a given step are due to the fact that we
can often  avoid unnecessary  work.  In  particular, $\URT$  takes as
an additional input the integers  $i,j,h$, with $h\leq i$. 
The integer  $i$ indicates the  current block
row  and $h$  the current  block  column, on  which to  mimic the  row
operations performed during  $\UR$ on block row $i$. 
If $j>1$ then we already computed some input $\K$ into which we need to 
riffle in zero (if $i\neq h$) or the relevant 
columns of the identity matrix (if $i = h $).
It turns out that it is easier to always riffle in zero  (the first line marked with $+$) 
and mimic the special case $i=h$ by adding the correct $0/1$ matrix to $V$ later in
the other line marked with $+$. 
If $j=1$ then we should initialise $\K$ with zero (if $i\neq h$) or  the relevant 
columns of the identity matrix (if $i = h $). As we only need the input $\K $ 
to define $V$ and $W$ in line (2), we mimic this by remembering that 
$W=0$ in this case and $V$ is either $0$ (if $i\neq h$) or the identity matrix
if $i=h$. So for $j=1 $ and $h=i$ we omit the multiplication
by the identity in lines (3) and (6). 

If  $i=1$ again Remark \ref{remread} applies accordingly to line (4). 
Note that due to the fact that $h\leq i$, this only happens if $h=i=1$.

In the following pseudo code describing Task $\URT$ we indicate in the last column 
which unstored matrices are implicitly known to be $0$ or the identity $1$.
\\

\begin{algorithm}[H]
\caption  {{\sc UpdateRowTrafo}\label{URT}}
\SetKwData{Left}{left}\SetKwData{This}{this}\SetKwData{Up}{up}
\SetKwFunction{Union}{Union}\SetKwFunction{FindCompress}{FindCompress}
\SetKwInOut{Input}{Input}\SetKwInOut{Output}{Output}
\SetKwInOut{Exput}{Exput}

\Input{$\A=(A,M,K,\rho',E,\lambda )$,  $\K \in \F^{\alpha \times \beta}$, 
	 $\M \in \F^{r\times \beta'}$, $\E=(\rho,\delta )$, $i, h, j$. } 
	 \Output{$\K\in \F^{(\alpha -r')\times (\beta+|\delta|)} $, $\M\in \F^{(r+r')\times \beta'}$.}
\BlankLine

         \begin{tabular}{l|lcl|l}
		 \hline 
		 & case & job & command & remark \\ 
		 \hline
$+$&  $j\neq 1$ & (CRZ): &  $\K := \CRZ(\K,\E.\delta)$; \\
&  $j= 1$ &  &  $-$ & \emph{$\K \mbox{ is } 0$}\\
  \hline
  $(1)$ & $h\neq i$, $j\neq1$ & (MAD): &  $Z := \K + \A.{\rm A}  \times  \M $; \\
&  $h\neq i$, $j= 1$ & (MUL): &  $Z := \A.{\rm A}  \times  \M $; &
  \emph{$\K\mbox{ is } 0$}\\
&  $h= i$, $j\neq 1$ & (CPY): &  $Z :=\K $; &
  \emph{$\M\mbox{ is } 0$}\\
&  $h= i$, $j= 1$ &  & $-$ &
  \emph{$Z\mbox{ is } 0$}\\
  \hline
  $(2)$&$\neg (j= 1 \wedge h=i)$ & (REX):  &  $V:=\A.\rho'  \times  Z$;  \\
&& &  $W:=\overline{\A.\rho'}\times Z$; \\
&  $j=1 \wedge h=i$& & $-$ & \emph{$V$, $W\mbox{ are } 0$}\\
  \hline
$+$ & $j\neq 1 \wedge h=i$ &  (ADI):  &  $V:={\rm ADI} (V,\E.\delta)$; \\ 
& $j=1 \wedge h=i $ & & $-$ & \emph{$V\mbox{ is } 1$}\\
    \hline
    $(3)$ & $\neg (j= 1 \wedge h=i)$    & (MUL):  &   $X:=\A.{\rm M} \times  V$; \\
    &  $j=1 \wedge h=i$& (CPY): & $X:= \A.{\rm M}$  & \emph{$V\mbox{ is } 1$}\\
    \hline
    $(4)$&$h\neq i$ & (MAD):  &  $S:=\M+\A.{\rm E} \times X$; \\
& $h=i\neq 1$ & (MUL):  &  $S:=\A.{\rm E} \times X$; & \emph{$M \mbox{ is } 0$}\\
& $h=i= 1$ &   &  $-$ & \emph{$S$  no rows}\\
    \hline
    $(5)$& $\neg (h=i=1)$ & (RRF):  &  $\M:=\RRF (\A.\lambda,S,X)$; \\
&     $h=i= 1$ & (CPY):  &  $\M:=X$; & \emph{$S $  no rows} \\
    \hline
    $(6)$& $\neg (j= 1 \wedge h=i)$    & (MAD):  &  $\K := W+\A.{\rm K} \times V $; \\
   &   $j=1 \wedge h=i$& (CPY): & $\K := \A.{\rm K}$; & \emph{$V \mbox{ is }  1$, $W\mbox{ is } 0$}\\
\end{tabular}
\end{algorithm}



\section{Concurrency analysis}\label{sec:concur}

To measure the degree of concurrency  of our algorithm we assign costs
to each of  the tasks.  We perform a relative  analysis, comparing the
cost  of a  parallel Gauss  algorithm with  the cost  of a  sequential
algorithm.  Therefore,  to simplify our  analysis, we assume  that the
cost  of a  matrix multiplication  of  $(\alpha \times  \beta )  \cdot
(\beta  \times \gamma  ) $  possibly followed  by addition  is $\alpha
\beta \gamma $ and the cost  of echelonising an $\alpha \times \beta $
matrix of  rank $r$  is $\alpha \beta  r $. It seems plausible that when
assuming that these costs are homogeneous functions of some degree,
bounded below by  $\omega$  (see \cite[Chapter~16]{Buerg}),
then in the results of Proposition~\ref{prop:concur} and
Theorem~\ref{the:concur} the degree of concurrency  can be replaced by some
constant times $a^{\omega-1}$. 
We  also assume  that all
blocks  are  square  matrices  of size  $\alpha\times  \alpha$,  where
$\alpha$ is not too small and
$$\alpha = \frac{n}{a} = \frac{m}{b}.$$
Then the  tasks {\sc  Extend}, {\sc RowLengthen},
{\sc  PreClearUp},  and {\sc  Copy} do not perform  any  time
consuming  operations  (compared to {\sc ClearDown} and 
{\sc UpdateRow}), so we  assign  cost  0 to these tasks.  

\begin{lemma}\label{closecosts}
 The cost of {\sc ClearDown} is 
bounded above by $\alpha^3$ 
and the cost of {\sc UpdateRow} is bounded above by  $1.25 \alpha^3$.
\end{lemma}

\begin{proof}
We start analysing the task {\sc ClearDown}: 
For $i=1$ only one job $\ech $ is performed contributing cost $\alpha^3$. 
Otherwise the first call of $\mad $ multiplies a matrix ${\bf A}.A$ of size 
$\alpha \times r $ with a matrix ${\bf D}.R $ of size $r \times (\alpha - r )$, 
contributing cost $\alpha r (\alpha - r)$. 
The echelonisation is done on an $\alpha \times (\alpha -r) $ matrix of rank $r'$ and the second $\mad $ multiplies an $r\times r'$ matrix by an 
$r' \times (\alpha -r -r') $ matrix. 
So in total the cost of {\sc ClearDown} is 
$$\alpha r (\alpha -r) + \alpha (\alpha -r) r' + rr'(\alpha - r -r' ) 
= \alpha ^2 (r+r') - rr'(\alpha + r + r' ) \leq \alpha ^3 $$ 
as $r+r' \leq \alpha $. 
\\
For the task {\sc UpdateRow} we similarly obtain the cost
$\alpha r \alpha $ for the $\mad $ in row (1), 
$r' r' \alpha $  for the $\mul $ in row (3), 
$r r' \alpha $ for the $\mad $ in row (4), and 
$(\alpha - r') r' \alpha $ for the $\mad $ in row (6). 
Summing up we obtain 
$$\alpha ^2 (r+r') + \alpha r r' \leq 1.25 \alpha^3 $$
again since $r+r' \leq \alpha $. 
\end{proof}

Ignoring all tasks of cost 0 Step~1 only involves the tasks {\sc ClearDown} and {\sc UpdateRow}.
The graph of task dependencies decomposes naturally into layers
according to the  value of $i+j$. 
We  abbreviate the call  $\CC(\C^{j}_{ij},\D^{i-1}_{j},i) $
with data  packages depending  on $i,j$ by  $\DA (i,j)$  and similarly
$\UR (\A_{ij}, \C^{j}_{ik},\B^{i-1}_{jk},i)$
by $\RA(i,j,k)$.
Then Figure~\ref{depgraph} displays the local task dependencies in
Step~1 for layers $i+j-1$ and $i+j$,  
where $k=j+1,\ldots , b$.\\

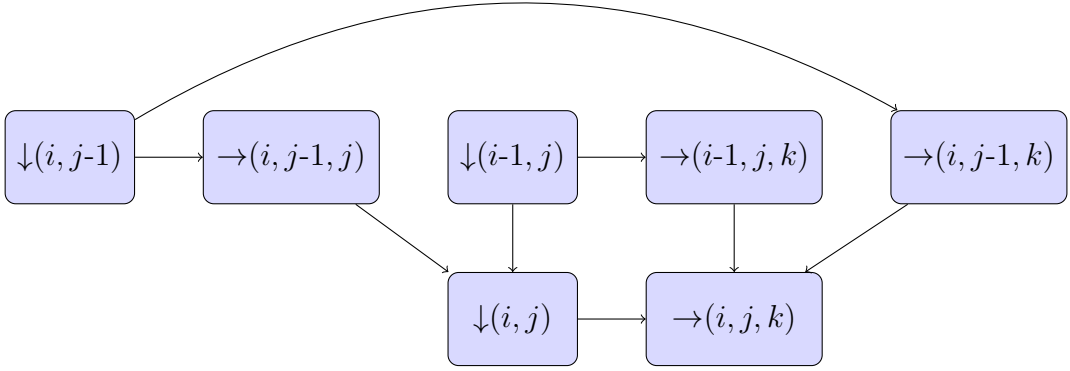
\begin{figure}
\begin{tikzpicture}[
  node distance = 0.9cm,
  auto,
  block/.style={rectangle, draw, fill=blue!15, 
    text width=5em, align=center, rounded corners, minimum height=3em},
  blockk/.style={rectangle, draw, fill=blue!15, 
    text width=3.5em, align=center, rounded corners, minimum height=3em},
  line/.style={draw, -latex'},
  cloud/.style={draw, ellipse,fill=red!20, node distance=3cm,
    minimum height=2em},
  ]
   \node [blockk] (init) {$\DA (i\mbox{-}1,j)$};
  
  \node [block, left=of init]   (URtilde1) {$\RA (i,j\mbox{-}1,j)$};
 \node [blockk, left=of URtilde1]   (URtilde2) {$\DA (i,j\mbox{-}1)$};
 \node [block, right=of init]   (UR1)  {$\RA (i\mbox{-}1,j,k)$};
  \node [block, right=of UR1]   (UR2) {$\RA (i,j\mbox{-}1,k)$};
  \node [blockk, below=of init]  (init2)  {$\DA (i,j)$};
  \node [block, right=of init2]  (URL2)  {$\RA (i,j,k)$};

  \path [line,->] (init)   -- (init2);
  \path [line,->] (URtilde1)  -- (init2);
  \path [line,->] (URtilde2)  -- (URtilde1);
  \path [bend left = 30,->] (URtilde2) edge (UR2);
  \path [line,->] (init)   -- (UR1);
  \path [line,->] (UR1)   -- (URL2);
  \path [line,->] (init2)  -- (URL2);
  \path [line,->] (UR2)  -- (URL2);
\end{tikzpicture}
\caption{Extract of task dependency graph}
\label{depgraph}
\end{figure}

Recall that a critical path in a task dependency graph is
the longest directed path between any pair of start node  and finish node. 
Its length is weighted by the cost of the nodes along the path. 

\begin{proposition} \label{length}
  \begin{enumerate}
    \item
The weighted length of a critical path in the task dependency graph of 
Step~1 is $2.25 \alpha^3(a+b-1)$.
    \item
The weighted length of a critical path in Step~3 is $\max ( (b-1) \alpha^3
, a \alpha^3 ) $.
  \end{enumerate}
\end{proposition}

\begin{proof}
 1.)  The task dependency graph splits naturally into 
$a+b-1$ layers according to the value of $i+j \in \{ 2,\ldots , a+b \} $. 
Within each layer, the critical paths involve 
{\sc ClearDown} and  {\sc UpdateRow} exactly once. 
So in total the length of a critical path is $(a+b-1) (\alpha^3 + 1.25 \alpha^3 ) $. 
\\
2.)
Step~3 only involves the task {\sc ClearUp} of non-zero cost $\alpha^3 $. 
The data package ${\bf R}_{jh} $ is only changed by the data packages below in
the same column so $h-j$ times. The maximum of 
$h-j$ is achieved at ${\bf R}_{1b}$  contributing the  cost  $\alpha^3 (b-1)$. 
For the transformation matrix, which can be done independently, 
 each of the ${\bf M}_{jh}$ is touched 
$a$ times contributing $a \alpha^3 $. 
\end{proof}

To determine the average degree of concurrency we divide the cost of the sequential 
{\sc Gauss}algorithm (with transformation matrix) applied to the $m\times n$-matrix $\CC$ 
by the weighted length of a critical path. For simplicity we assume that $m=n$, $a=b$ and that all blocks are of the same size $\alpha \times \beta $ with
$\alpha = \beta $.

\begin{proposition}  \label{prop:concur}
  Under the assumptions above   and using Lemma \ref{closecosts} 
	the average degree of concurrency of {\sc the Chief} is 
$\frac{1}{5.5} a^{2} $.
\end{proposition}

\begin{proof}
	By our assumptions $n=m$ and the cost of the sequential {\sc Gauss}algorithm (with transformation matrix) applied to the huge matrix $\CC $ is 
$n^3 = a^{3 } \alpha^3$.
By Proposition \ref{length}
the weighted length of a critical path in 
the complete algorithm {\sc the Chief} is 
$(2.25 (2a-1) + a) \alpha^3 = (5.5a-2.25) \alpha^3 \sim 5.5 a \alpha^3 $. 
\end{proof}

In practical examples the gain of performance is much better.  This is
partly due to the  fact that we split our matrix  into blocks that fit
into  the computer's  memory; an  echelonisation of  the huge  matrix,
however, would require to permanently read  and write data to the hard
disk.  The  other reason is  that in random  examples the length  of a
critical path  is much shorter.   To make  this precise we  assume, in
addition to  the assumptions  above of starting  with a  square matrix
partitioned into square blocks of  equal size $\alpha$, that our input
matrix  is {\em  well-conditioned},  by  which we  mean  that the  $a$
top-left square  submatrices of the input  matrix of size $j  \alpha $
($j=1,\ldots , a  $) have full rank.  Then,  properly implemented, the
cost of {\sc ClearDown} is $\alpha^3 $,  if it is called for $i=j$ and
$0$ otherwise.  Also in  {\sc Update Row} $r' = 0$ and  so the cost of
{\sc Update Row}  is $\alpha^3 $ (this can be  shown without using the
assumptions  in   Lemma  \ref{closecosts}).   In  particular   in  the
dependency graph above, the weighted length  of a critical path in any
odd  layer  is $\alpha^3  $  and  in an  even  layer,  this length  is
$2\alpha^3 $  (resp. $\alpha^3 $  for the  last layer). In  total this
shows the following

\begin{theorem}\label{the:concur}
For a well-conditioned square matrix, the length of a critical path in
the dependency graph is $\alpha^3 (3a-2)$ and hence the 
average degree of concurrency of {\sc the Chief} is
$\frac{1}{3} a^{2}$.
\end{theorem}

\begin{remark}
The concurrency analysis above is not sufficient to ensure the effective
use of all processors throughout the entire run that we see in the
experimental results below.

Although the steps are estimated at  their worst case in practice many
tasks early  in the task dependency  graph execute a lot
faster. Provided there are sufficiently many blocks, 
enough work is available for  execution considerably earlier than
the task dependency graph might suggest. Assigning a priority $i+j$ to
a  task pertaining  to the  $i$-th row  and $j$-th  column directs  an
appropriate scheduling.
\end{remark}

\section{Experimental results}\label{sec:exper}

We give timings for the two implementations mentioned in
Section~\ref{sec:model}.

\subsection*{The Meataxe}

In order to  demonstrate the power of this algorithm the following
tests were done on   a  machine  with
64-piledriver cores with 512 GB of memory  running at $2.7$ gigaherz.

We chose as our
first  example  a  random  $1,000,000  \times  1,000,000$-matrix  with
entries in  the field  of order  2.  To put  this matrix  into reduced
echelon form  with transformation matrix  we chose to chop  the matrix
into  blocks  of  size  $50,000\times  50,000$. This run
took 520min.

The following graph  shows the progress of the calculation. The red
shows that over 60 cores were used for the vast majority of the
time. The blue shows that during Steps~1  and~2 the memory footprint was
fairly constant but at the transition to Step~3 about 30\% more memory
was needed, due mainly to the expansion of the matrix ${\mathcal M}.$

\begin{center}
\pgfuseimage{rla1}
\end{center}

For a  second example, on the same machine, we echelonised with
transformation matrix a 
random  $600,000  \times  600,000$-matrix  with
entries in  the field  of order  3 in 460 min, using a block size of $30,000.$
We do not give a graph for this run, as it is almost indistinguishable
from the one given above.

The above examples were done with a carefully chosen block size.
The following two examples highlight the effect of too large a block
size. We echelonised the same 
random  $300,000  \times  300,000$-matrix  with entries in the field
of order 3 using block sizes of $30,000$ and $15,000$, respectively.
In the first graph we see that 
$15,000$ is again a good choice of block size. Almost all cores are used
for most of the time. 
In the second graph we see that 
$30,000$ is  too large a block size, so that the 64 available cores
are seldomly utilised at once, leading to a run time which is about
$50$\% longer.
\begin{center}
\pgfuseimage{ch20} 
\pgfuseimage{ch10} 
\end{center}

Note that in general terms the necessary  block size agrees with the
concurrency analysis. 

Our final three examples are intended to
demonstrate that our methods are not restricted
to tiny fields, nor to prime fields.
A random  $200,000  \times  200,000$-matrix  with entries in the field
with $193$ took 445 mins,
a random  $200,000  \times  200,000$-matrix  with entries in the field
with $1331=11^3$ took 615 mins, and
a random  $100,000  \times  100,000$-matrix  with entries in the field
$50653=37^3$ elements took 200 mins.



\section{Acknowledgements}

We thank  Martin Albrecht for discussions  in the early stages  of the
algorithm design.  The first full implementation of this algorithm was
developed by Jendrik  Brachter in {\sf GAP} \cite{GAP} which was
essential to getting the details of the algorithm right.
A  parallel version of this implementation, using
{\sf HPC-GAP},  was produced by Jendrik  Brachter and Sergio
Siccha. 

The second and third author acknowledge the support of the SFB-TR 195.

 \end{document}